\documentclass[reqno,11pt]{template}
\usepackage{amsfonts}
\usepackage{amsthm,amsmath,eucal,bbm}
\usepackage[integrals]{wasysym}
\usepackage{enumerate}
\usepackage{fancyhdr}
\usepackage{amssymb}
\usepackage{stmaryrd}
\usepackage{times}
\usepackage[body={34cc,49cc},centering]{geometry}
\usepackage{appendix}

 \renewcommand{\normalsize}{\fontsize{11pt}{1.1\baselineskip}\selectfont}
 \DeclareMathSizes{11}{11}{7}{5}

\fancypagestyle{plain}{%
  \fancyhead{} 
}

 \newtheorem{theorem}{Theorem}[section]
 \newtheorem{corollary}[theorem]{Corollary}
 \newtheorem{lemma}[theorem]{Lemma}
 \newtheorem{proposition}[theorem]{Proposition}

 \theoremstyle{definition}
 
 \theoremstyle{definition}

 \newtheorem{remark}[theorem]{Remark}
 \theoremstyle{remark}

\newcommand{\eps}{\varepsilon}

\newcommand{\di}{\,d}

\newcommand{\norm}[1]{\big\Vert#1\big\Vert}

\newcommand{\abs}[1]{\left\vert#1\right\vert}
\newcommand{\set}[1]{\left\{#1\right\}}
\newcommand{\bigset}[1]{\big\{#1\big\}}
\newcommand{\inner}[1]{\left(#1\right)}
\newcommand{\biginner}[1]{\Bigl(#1\Bigr)}

\newcommand{\comii}[1]{\left<#1\right>}
\newcommand{\com}[1]{\bigl[#1\bigr]}
\newcommand{\reff}[1]{(\ref{#1})}

\begin{document}

\title[]
{ Global hypoelliptic estimates for a linear model of\\[2pt] non-cutoff    Boltzmann  equation}

\author[Wei-Xi Li]{Wei-Xi Li}



\date{}

\address{Wei-Xi Li\newline
 Laboratoire de Math\'{e}matiques Jean Leray,
Universit\'{e} de Nantes, 44322 Nantes,
France\newline
and\newline
School of Mathematics and Statistics, Wuhan University, 430072
Wuhan, China}
\email{Wei-Xi.Li@univ-nantes.fr}

\thanks{The author gratefully acknowledges the support from the Project NONAa of  French (No. ANR-08-BLAN-
0228-01)  and the NSF of China (No. 11001207)}
\begin{abstract}
  In this paper we study a linear  model of  spatially inhomogeneous Boltzmann equation without angular cutoff. Using the  multiplier method introduced by F. H\'{e}rau and K. Pravda-Starov (2011), we establish the optimal global hypoelliptic estimate with weights for the linear model operator.
\end{abstract}


\maketitle 

\section{Introduction and main results}

Inspired by the   work of H\'{e}rau and Pravda-Starov \cite{HK2011}  on the global hypoellipticity of  Landau-type operator,  we study in this paper  the  hypoellipticity    of a linear model of spatially inhomogeneous Boltzmann  equation without angular cutoff, which takes the following form:
\begin{eqnarray}\label{11051116}
     \mathcal P=\partial_t +v\cdot \partial_x  +a(v)(-\tilde \triangle_v)^{s} +b(v), \quad 0<s< 1,
\end{eqnarray}
where  the coefficients $a$, $b$ are smooth {\it real-valued}
functions  of  the velocity variable $v$  with the properties subsequently listed below. There exist  a number $\gamma\in\mathbb R$  and a constant $C\geq1$ such that for all $v\in \mathbb R^n$ we have
  \begin{eqnarray}\label{assumption1}
      C^{-1} \comii v^\gamma\leq a(v)\leq
    C\comii{v}^{2s+\gamma},\quad C^{-1} \comii v^{2s+\gamma}\leq b(v)\leq
    C\comii{v}^{2s+\gamma},
  \end{eqnarray}
  and
  \begin{eqnarray}\label{assumption2}
    \forall~ \abs \alpha\geq0,~\exists~C_\alpha>0, \quad \abs{\partial_v^\alpha  a(v)}+
    \abs{\partial_v^\alpha b(v)}\leq
    C_\alpha\comii{v}^{2s+\gamma-\abs\alpha},
  \end{eqnarray}
where and  throughout the paper we use the notation $\comii{\cdot}=\inner{1+\abs {\cdot}^2}^{1/2}$. The notation
$(-\tilde\triangle_v)^s$ in \reff{11051116} stands for the Fourier multiplier of
symbol
$$ \abs\eta^{2s}\omega(\eta)+\abs\eta^2(1-\omega(\eta)),$$
 with $\omega(\eta)\in C^\infty(\mathbb R^n;~[0,1])$, such that
 $\omega=1$ if $\abs\eta\geq 2$ and $\omega=0$ if $\abs\eta\leq1$. Here $\eta$ is the dual variable of $v$.

Let's first explain the motivation for studying such a kind of operator  $\mathcal P$,  which is     closely linked with the spatially inhomogeneous Boltzmann equation which has singularity in both the kinetic part and the angular part. Precisely,   non-cutoff Boltzmann equation in $\mathbb R^n$ reads
\begin{eqnarray}\label{11051116+}
  \partial_t f+v\cdot\partial_x f=Q(f,f),
\end{eqnarray}
where  $f(t,x,v)$ is a real-valued function, standing for the time-dependent probability  density of particles with velocity $v$ at position $x$.  The  right hand side of \reff{11051116+} is the Boltzmann bilinear collision  operator which acts only on the velocity variable $v$  by
\begin{eqnarray*}
  Q(g,f)(v)=\int_{\mathbb R^n}\int_{\mathbb S^{n-1}} B(\abs{v-v_*}, \sigma) \inner{g_*'f'-g f}\di v_* d\sigma.
\end{eqnarray*}
Here we use the shorthand $f=f(t,x,v)$, $f_*=f(t,x,v_*)$, $f'=f(t,x,v')$, $f_*'=f(t,x,v_*')$, and for $\sigma\in\mathbb S^{n-1}$,
\[
   v'=\frac{v+v_*}{2}+\frac{\abs{v-v_*}}{2}\sigma,\quad v_*'=\frac{v+v_*}{2}-\frac{\abs{v-v_*}}{2}\sigma.
\]
In the above relations, $v'$, $v_*'$ and $v$, $v_*$ are the velocities of a pair of  particles before and
after collision. The collision cross-section $B(\abs{v-v_*}, \sigma)$ is a non-negative function which only depends on the
relative velocity $\abs{v - v_*}$ and the deviation angle $\theta$  through $\cos\theta=\frac{v-v_*}{\abs{v-v_*}}\cdot\sigma$, and
takes the following form
\begin{eqnarray*}
  B(\abs{v-v_*}, \sigma)=\Phi\inner{\abs{v-v_*}} b\inner{\cos\theta},\quad \cos\theta=\frac{v-v_*}{\abs{v-v_*}}\cdot\sigma,~\,0\leq \theta\leq{\pi\over2},
\end{eqnarray*}
where the kinetic part 	$\Phi$ is given by
\[
   \Phi\inner{\abs{v-v_*}}=\abs{v-v_*}^\gamma,\quad\gamma>-3,
\]
and the  angular part $b$ satisfies, with $0<s<1$,
\begin{eqnarray*}
  b\inner{\cos\theta}\approx \theta^{-(n-1)-2s} \,~{\rm as} ~\,\theta\rightarrow 0.
\end{eqnarray*}
We refer to  \cite{MR2556715, AC02,  CerIllPul94, DL8901, Villani98-2}   and the   references therein   for the physical background and derivation of the Boltzmann  equation, as well as the mathematical theory on the Boltzmann equation.   Note that the angular cross-section $b$ is not integrable on the sphere due to the singularity  $\theta^{-(n-1)-2s}$, which leads to the conjecture that the nonlinear collision operator should behave like a fractional Laplacian; that is,
\begin{eqnarray*}
    Q(g,f)\approx - C_g(-\triangle_v)^sf+\textrm{lower order terms},
\end{eqnarray*}
with $C_g>0$ a constant depending only on the physical properties of $g$.   Initiated by  Desvillettes \cite{Desv97,MR1324404},  there have been extensive works which give partial support to the conjecture regarding the  smoothness of solutions for the homogeneous Boltzmann equation without angular cutoff,  c.f.  \cite{ADVW00,AS09, MR2149928, MR2557895, DesvVillani00-2, DesvWennberg04, MR2425608, MR2476686, MR2523694}.   For the inhomogeneous case the study becomes more complicated, due to the coupling of the transport operator  with the collision operator, and the commutator between pseudo-differential operators and the collision operator.  Recent works \cite{AMUXY3, AMUXY2, AMUXY1, MR2679369,MR2462585,GS1102,GS1101,Lions98, Mou2,Mou1}  indicate the linearized Boltzmann operator around a normalized Maxwellian distribution behaves essentially like the operator given in \reff{11051116}.  To explain it more precisely,  let's  first recall the linearization process.  Denote by $\mu$ the normalized Maxwellian distribution; that is
\[
   \mu(v)=\inner{2\pi}^{-n/2}e^{-\abs{v}^2/2}.
\]
By setting $f=\mu+\sqrt{\mu}g$,  we see the perturbation $g$ satisfies the equation
\[
    \partial_t  g+v\cdot\partial_x g-\mu^{-1/2}Q(\mu,~\sqrt{\mu}g)-\mu^{-1/2}Q(\sqrt{\mu}g,~\mu)=\mu^{-1/2}Q(\sqrt{\mu}g,~\sqrt{\mu}g),
\]
since $\partial_t  f+v\cdot\partial_x f-Q(f,~f)=0$ and $Q(\mu,~\mu)=0$. Using the notation
\[
  \Gamma(g,~h)=\mu^{-1/2}Q(\sqrt{\mu}g,~\sqrt{\mu}h),
\]
we may rewrite the above equation as
\begin{eqnarray*}
  \partial_t g +v\cdot\partial_x g-\Gamma(\sqrt{\mu},~g)-\Gamma(g,~\sqrt{\mu})=\Gamma(g,~g).
\end{eqnarray*}
Due to the following coercivity and upper bound estimates established in \cite{AMUXY1}, with $H^m(\mathbb R_v^n)$, $m\in\mathbb R$, the usual Sobolev space,
\begin{eqnarray*}
    C^{-1}\inner{\norm{\comii v^{\gamma\over2}g}_{H^s(\mathbb R_v^n)}^2+\norm{\comii v^{s+{\gamma\over2}} g}_{L^2(\mathbb R_v^n)}^2}\leq   \biginner{-\Gamma(\sqrt{\mu},~g)-\Gamma(g,~\sqrt{\mu}),~g}_{L^2(\mathbb R_v^n)} +\norm{g}_{L^2(\mathbb R_v^n)}^2
  \end{eqnarray*}
  and
  \begin{eqnarray*}
    \biginner{-\Gamma(\sqrt{\mu},~g)-\Gamma(g,~\sqrt{\mu}),~g}_{L^2(\mathbb R_v^n)} \leq C \norm{\comii v^{s+{\gamma\over2} } g}_{H^s(\mathbb R_v^n)}^2,
  \end{eqnarray*}
we see that  the linear part $-\Gamma(\sqrt{\mu},~g)-\Gamma(g,~\sqrt{\mu})$  of the Boltzmann collision operator
behaves like a generalized Kolmogorov type operator
\[
     \partial_t+v\cdot \partial_x+a(v)(-\tilde\triangle_v)^{s} +b(v),
\]
with $a(v)$, $b(v)$ satisfying the conditions  \reff{assumption1} and \reff{assumption2}. This motivates the present work on the global hypoellipticity of the  operator $\mathcal P$ given in \reff{11051116}.

We remark that there have been some related works  concerned with a linear model of spatially  inhomogeneous Boltzmann equation, which takes the following form
\begin{eqnarray}\label{11053001}
  P=\partial_t   +v\cdot\partial_x   -\tilde a(t,x,v)(-\tilde \triangle_v)^s,\quad \inf_{t,x,v}~ \tilde a(t,x,v)>0, ~~\tilde a\in C^\infty_b,
\end{eqnarray}
where $ C^\infty_b$ stands for the space of smooth functions whose derivatives of any order are bounded.  As far as we know, the model operator \reff{11053001} was firstly studied by Morimoto and Xu \cite{MorimotoXu07b} for $1/3<s \leq 1$, and then  was  improved by Chen et al. \cite{MR2763329} by virtue of Kohn's method. Recently  Lerner et al. \cite{LMP} established optimal results using the Wick quantization techniques \cite{MR2477145, MR1957713}, and then a  simpler proof  was presented by  Alexandre \cite{A1} following the ideas of  Bouchut \cite{Bouchut02} and Perthame \cite{Per1},  completing  the study of the operator $P$ given in \reff{11053001}.    However these works are mainly concerned with the local hypoelliptic estimates in the sense that the coefficient $\tilde a$ in \reff{11053001} has strictly positive lower bound and bounded derivatives.  Compared  with the  operator   in \reff{11053001}, our model operator  $\mathcal P$ in \reff{11051116} is closer to the linearized Boltzmann equation in view of the  aforementioned  coercivity estimate and upper bound estimate.  Moreover we do not need the restrictions that $\inf_{t,x,v} \tilde a(t,x,v)>0$ and $\tilde a\in C^\infty_b$, since the coefficients in \reff{11051116}  may trend  to $0$ or $+\infty$ as $\abs v\rightarrow +\infty$,  depending on the sign of $\gamma$.

Now we state our main results as follows.

\begin{theorem}\label{th1}
   let $\mathcal P$ be given in \reff{11051116} with $a(v)$, $b(v)$ satisfying the conditions  \reff{assumption1} and \reff{assumption2}.  Then for all  $m\in\mathbb R$, there exists a constant $C_m$ such that  for all $f\in \mathcal S\inner{\mathbb R^{2n+1}}$ we have
  \begin{eqnarray*}
  \begin{split}
        &\norm{\comii{v}^{\frac{\gamma-2s}{1+2s}}\abs{D_t}^{\frac{2s}{1+2s}}f}_{H^m}+
        \norm{\comii{v}^{\frac{\gamma}{1+2s}}\abs{D_x}^{\frac{2s}{1+2s}}f}_{H^m}+\norm{\comii{v}^{\gamma}\comii{D_v}^{2s}f}_{H^m}+
        \norm{\comii{v}^{2s+\gamma}f}_{H^m}\\
        &\leq  C_{m}\biginner{
     \norm{{\mathcal P} f}_{H^m}+\norm{f}_{H^m}},
  \end{split}
  \end{eqnarray*}
   where $\norm{\cdot}_{H^m}$  stands for $\norm{\cdot}_{H^m\inner{\mathbb R_{t,x,v}^{2n+1}}}$, and $D_t=\frac{1}{i}\partial_t$, $D_x=\frac{1}{i}\partial_x$, etc.
\end{theorem}

\begin{remark}
  It seems that the multiplier method used in the paper can also be  applied  to the linearized Boltzmann operator $L$ given by
  \[
     Lg=\partial_t g +v\cdot\partial_x g-\Gamma(\sqrt{\mu},~g)-\Gamma(g,~\sqrt{\mu}),
  \]
  and gives the same hypoellipticity as above.  But the situation is more complicated, and we should pay more attention to handling the commutators between $L$ and pseudo-differential operators. We hope to study this issue in a future work.
\end{remark}

We end up the introduction by a few comments on the  exponents of derivative terms and weight terms  in   Theorem \ref{th1}. These   exponents  seem  to be  optimal.  When restricted to a fixed compact subset $K\subset\mathbb R^{2n+1}$, instead of the whole space,  the  problems reduce  to a local version, and the operator becomes the type  given in \reff{11053001}, for which the exponent $2s/(2s+1)$ for the regularity in the time
and space variables is indeed sharp by using a simple scaling argument (see \cite{LMP} for more detail).  In the particular case when $s=1$, we have  a type of differential operator,  which seems simpler to handle than fractional derivatives,  and our exponents in the regularity terms and weight terms coincide well with the ones in \cite{HK2011}.

\section{Notations and estimates on commutator  with pseudo-differential operators}

\subsection{Notations and some basic facts on symbolic calculus}

Notice that the diffusion term  in \reff{11051116} is   an operator  only with respect to the velocity variable $v$. So it is convenient to  take  partial
Fourier transform in the $t, x$ variables,  and then to study the  operator  on the Fourier side
\begin{eqnarray}\label{11051401}
  \tilde{\mathcal P}=i\inner{\tau+v\cdot\xi}+a(v)(-\tilde \triangle_v)^{s} +b(v),
\end{eqnarray}
 where and throughout the paper, $(\tau, \xi)$ always stand for the dual variables of $(t,x)$ and are considered as parameters, while  $\eta$ will be used to denote the dual variable of $v$. Since our analysis is on $\mathbb R_v^n$,   we will use   $\inner{\cdot, \cdot}_{L^2}$  and  $\norm{\cdot}_{L^2}$, instead of  $\inner{\cdot, \cdot}_{L^2(\mathbb R_v^n)}$  and  $\norm{\cdot}_{L^2(\mathbb R_v^n)}$, to denote the inner product and  norm  in $L^2(\mathbb R_v^n)$, if no confusion occurs.

  To simplify the notation, by $A\lesssim B$ we mean there exists a positive harmless constant  $C>0$ such that $A\leq C B$, and similarly for $A\gtrsim B$. While the notation $A\approx B$ means both $A\lesssim B$ and $B\lesssim A$ hold.

Now we recall some basic facts on symbolic calculus, and refer to Chapter 18 of \cite{Hormander85} and \cite{MR2599384} for detailed discussion on the pseudo-differential calculus.   In the sequel discussion, let  $m(v,\eta)$ be an admissible weight with respect to the constant metric $\abs{dv}^2+\abs{d\eta}^2$. By admissible weight we mean  that
\begin{eqnarray*}
  \exists~ C>0, r>0,~~\,\forall~~(v,\eta), (\tilde v,\tilde \eta)\in\mathbb R^{2n},\quad\abs{(v,\eta)-(\tilde v,\tilde \eta)}\leq r\Longrightarrow  C^{-1}\leq \frac{m (v,\eta) }{m( \tilde v,\tilde \eta )}\leq C,
\end{eqnarray*}
and  that
\begin{eqnarray*}
  \exists~ C>0, N>0,~~\,\forall~~(v,\eta), (\tilde v,\tilde \eta)\in\mathbb R^{2n},\quad \frac{m(v,\eta)}{m(\tilde v,\tilde \eta)}\leq C\biginner{1+\abs{(v,\eta)-(\tilde v,\tilde \eta)}}^N.
\end{eqnarray*}    Consider a symbol $p(\tau,\xi, v,\eta)$ as a function of $(v,\eta)$ with parameters $(\tau, \xi)$, and we say $p\in S\inner{m,~\abs{dv}^2+\abs{d\eta}^2}$ uniformly with respect to $(\tau, \xi)$, if
\[
    \forall~ \alpha, \beta\in\mathbb Z_+^n,~~\forall~v,\eta\in\mathbb R^n,\quad  \abs{\partial_v^\alpha\partial_\eta^\beta p(\tau,\xi,v,\eta)}\leq C_{\alpha,\beta}~ m(v,\eta),
\]
with $ C_{\alpha,\beta}$ a constant depending only on $\alpha$ and $\beta$, but independent of $(\tau, \xi)$. For simplicity of notations,  we will omit   the parameters $(\tau, \xi)$ in   symbols,   and by {\it  $p\in S\inner{m,~\abs{dv}^2+\abs{d\eta}^2}$ we always mean that $p$  satisfies the above inequality uniformly with respect to} $(\tau, \xi)$. Denote by
\[{\rm Op} \inner{S(m,~\abs{dv}^2+\abs{d\eta}^2)}\]
the class of pseudo-differential operators $p^w$ with $p\in S\inner{m,~\abs{dv}^2+\abs{d\eta}^2}$. Here $p^w$ stands for the Weyl quantization of symbol $p$, defined by
\[
  p^w u(v)=\int_{\mathbb R^6}e^{2i\pi(v-z)\cdot\eta} p\inner{\frac{v+z}{2},~\eta}u(z)\di zd\eta.
\]
One of the elementary properties of the Weyl quantization is the boundedness in $L^2$ of the operator $p^w$ with $p\in S\inner{1,~\abs{dv}^2+\abs{d\eta}^2}$.  If $p_i\in S\inner{m_i,~\abs{dv}^2+\abs{d\eta}^2}$, $i=1,~2$, then we have (see Theorem 2.3.8 of \cite{MR2599384} for instance)
\begin{eqnarray}\label{11053010}
  p_1^wp_2^w\in {\rm Op}\inner{S\inner{m_1m_2,~\abs{dv}^2+\abs{d\eta}^2}}.
\end{eqnarray}
In view of \reff{assumption2}, symbolic calculus  (Theorem 2.3.8 and Corollary 2.3.10 of \cite{MR2599384}) shows that for any $m\in\mathbb R$ and any $\ell\in\mathbb R$ we have
 \begin{eqnarray}\label{11060401}
   \com{\comii{D_\eta}^m,~\comii v^\ell} \in {\rm Op}\inner{S\inner{\comii v^{\ell-1}\comii\eta^{m-1},~\abs{dv}^2+\abs{d\eta}^2}},
\end{eqnarray}
 \begin{eqnarray}\label{11060201}
   \com{\comii{D_\eta}^m,~a},\,\,~\com{\comii{D_\eta}^m,~b}  \in {\rm Op}\inner{S\inner{\comii v^{2s+\gamma-1}\comii\eta^{m-1},~\abs{dv}^2+\abs{d\eta}^2}},
\end{eqnarray}
and
\begin{eqnarray}\label{11060202}
  \com{p^w,~a},~\,\,\com{p^w,~b}\in   {\rm Op}\inner{S\inner{\comii v^{2s+\gamma-1},~\abs{dv}^2+\abs{d\eta}^2}},
\end{eqnarray}
where $p\in S\inner{1,~\abs{dv}^2+\abs{d\eta}^2}$ and $[A, B]$ stands for the commutator between $A$ and $B$   defined by $[A, B]=AB-BA$.

\begin{lemma}\label{lem11053101}
  Let $\tilde{\mathcal P}$ be given in \reff{11051401} with $a$, $b$ satisfying the assumptions \reff{assumption1} and \reff{assumption2}.   Then   for  all $f\in\mathcal S(\mathbb R_v^n)$ we have
  \begin{eqnarray}\label{11050715}
    \norm{a^{1\over 2}(-\tilde\triangle_v)^{s/2}f}_{L^2 }^2+\norm {\comii{D_\eta}^s \comii v^{\gamma/2}f}_{L^2}^2+\norm{\comii v^{s+\gamma/2}f}_{L^2 }^2\lesssim \abs{\biginner{\tilde{\mathcal P} f,~f}_{L^2 }}+\norm{ f}_{L^2 }^2.
 \end{eqnarray}
\end{lemma}

\begin{proof}
  We only need to treat the first and   third terms on the left hand side of \reff{11050715}, since by \reff{assumption1} and \reff{11060401} one has
  \[
      \norm {\comii{D_\eta}^s \comii v^{\gamma/2}f}_{L^2}\lesssim \norm{a^{1\over 2}(-\tilde\triangle_v)^{s/2}f}_{L^2 }+\norm{\comii v^{s+\gamma/2}f}_{L^2 }+\norm{f}_{L^2 }.
  \]
  Let $f\in \mathcal S\inner{\mathbb R_v^n}$.  Observe the operator $i\inner{\tau+v\cdot\xi}$ is skew-adjoint.  It then follows that
  \begin{eqnarray*}
    {\rm Re}\biginner{a(v)(-\tilde\triangle_v)^{s}f,~f}_{L^2 }+{\rm Re} \biginner{b(v)~f,~f}_{L^2 }= {\rm Re}\biginner{\tilde{\mathcal P} f,~f}_{L^2 }.
  \end{eqnarray*}
  Note that $a,~b$ are real-valued functions. Then by virtue of the relation that
  \[
      {\rm Re}\biginner{a(v)(-\tilde\triangle_v)^{s}f,~f}_{L^2 }= \biginner{(-\tilde\triangle_v)^{s\over2} a(v)  (-\tilde\triangle_v)^{s\over2} f,~f}_{L^2 }-{\rm Re}\biginner{\com{(-\tilde\triangle_v)^{s\over2},~a }(-\tilde\triangle_v)^{s\over 2}f,~f}_{L^2 },
  \]
  we have
  \begin{eqnarray}\label{11053101}
     \norm{a^{1/2}(-\tilde\triangle_v)^{s/2}f}_{L^2 }^2 +\norm{b ^{1/2}f}_{L^2 }^2\lesssim
     \abs{\biginner{\tilde{\mathcal P} f,~f}_{L^2 }}+\abs{\biginner{\com{(-\tilde\triangle_v)^{s\over2},~a }(-\tilde\triangle_v)^{s\over 2}f,~f}_{L^2 }}.
  \end{eqnarray}
  In view of  \reff{11060201}, we see
  \begin{eqnarray*}
    \com{(-\tilde\triangle_v)^{s\over2},~a }\in{\rm Op} \inner{S\inner{\comii v^{2s+\gamma-1},~\abs{dv}^2+\abs{d\eta}^2}},
  \end{eqnarray*}
  and thus by \reff{11053010}
  \[
  \comii v^{-(2s+{\gamma\over2}-1)}\com{(-\tilde\triangle_v)^{s\over2},~a }\comii v^{-\gamma/2}\in{\rm Op} \inner{S\inner{1,~\abs{dv}^2+\abs{d\eta}^2}}.
  \]
  As a result, writing
  \[
     \com{(-\tilde\triangle_v)^{s\over2},~a }=\comii v^{(2s+{\gamma\over2}-1)}\inner{\comii v^{-(2s+{\gamma\over2}-1)}\com{(-\tilde\triangle_v)^{s\over2},~a }\comii v^{-\gamma/2}}\comii v^{\gamma/2},
  \]
  we have
  \begin{eqnarray*}
    \abs{\biginner{\com{(-\tilde\triangle_v)^{s\over2},~a }(-\tilde\triangle_v)^{s\over 2}f,~f}_{L^2 }} \lesssim \norm{\comii v^{\gamma/2}(-\tilde\triangle_v)^{s/2}f}_{L^2 }\norm{ \comii v^{2s+{\gamma\over 2}-1}f}_{L^2 }.
  \end{eqnarray*}
   This along with  the interpolation inequality
   \[
      \norm{ \comii v^{2s+{\gamma\over 2}-1}f}_{L^2 }\lesssim \eps \norm{ \comii v^{s+{\gamma\over 2}}f}_{L^2 }+C_\eps \norm{ f}_{L^2 }
   \]
   due to the fact that $s<1$, gives
   \begin{eqnarray*}
    \abs{\biginner{\com{(-\tilde\triangle_v)^{s\over2},~a }(-\tilde\triangle_v)^{s\over 2}f,~f}_{L^2 }} &\lesssim& \eps \norm{\comii v^{\gamma/2}(-\tilde\triangle_v)^{s/2}f}_{L^2 }^2+\eps \norm{ \comii v^{s+{\gamma\over 2}}f}_{L^2 }^2+C_\eps \norm{ f}_{L^2 }^2\\
    &\lesssim& \eps \norm{a^{1/2}(-\tilde\triangle_v)^{s/2}f}_{L^2 }^2+\eps \norm{ b^{1/2}f}_{L^2 }^2+C_\eps \norm{ f}_{L^2 }^2,
  \end{eqnarray*}
  where the last inequality follows from \reff{assumption1}.  Combining  \reff{11053101} we conclude
  \[
     \norm{a^{1\over 2}(-\tilde\triangle_v)^{s/2}f}_{L^2 }^2 +\norm{b ^{1\over 2}f}_{L^2 }^2\lesssim
       \eps \norm{a^{1\over 2}(-\tilde\triangle_v)^{s/2}f}_{L^2 }^2+\eps \norm{ b^{1\over 2}f}_{L^2 }^2+\abs{\biginner{\tilde{\mathcal P} f,~f}_{L^2 }}+C_\eps \norm{ f}_{L^2 }^2.
  \]
  Taking $\eps$ sufficiently small gives the desired estimate \reff{11050715}, completing the proof of Lemma \ref{lem11053101}.
\end{proof}

\begin{corollary}
  Let $\tilde{\mathcal P}$ be given in \reff{11051401} with $a$, $b$ satisfying the assumptions \reff{assumption1} and \reff{assumption2}, and let  $p\in S(1, \abs{dv}^2+\abs{d\eta}^2)$. Then
  \begin{eqnarray}\label{11051518}
   \forall~ f\in\mathcal S(\mathbb R_v^n),\quad \abs{\biginner{a(v)(-\tilde\triangle_v)^{s}f+b~ f,~~ p^w f}_{L^2}}\lesssim \abs{  \biginner{\tilde{\mathcal P}f,  f}_{L^2}}+\norm{f}_{L^2}^2.
  \end{eqnarray}
\end{corollary}

\begin{proof}
  In view of \reff{assumption1},  it is clear that
  \begin{eqnarray*}
   \abs{\biginner{b~ f,~~ p^w f}_{L^2}}\lesssim \norm{\comii v^{s+\gamma/2}f}_{L^2}\norm{\comii v^{s+\gamma/2}p^wf}_{L^2}\lesssim \norm{\comii v^{s+\gamma/2}f}_{L^2}^2,
  \end{eqnarray*}
  where the last inequality holds because
  \[
    \norm{\comii v^{s+\gamma/2}p^wf}_{L^2}\lesssim \norm{p^w \comii v^{s+\gamma/2}f}_{L^2}+\norm{\com{p^w,~ \comii v^{s+\gamma/2}}f}_{L^2}\lesssim \norm{\comii v^{s+\gamma/2}f}_{L^2},
  \]
  since $p\in S(1, \abs{dv}^2+\abs{d\eta}^2)$.  By virtue of \reff{11050715}, the desired estimate \reff{11051518} will follow if we could show that
    \begin{eqnarray}\label{11053110}
    \abs{\biginner{a(v)(-\tilde\triangle_v)^{s}f,~~ p^w f}_{L^2}}\lesssim \norm{ a^{1\over2 }(-\tilde\triangle_v)^{s/2}f}_{L^2}^2+\norm {\comii{D_\eta}^s \comii v^{\gamma\over2}f}_{L^2}^2+\norm{ \comii v^{s+{\gamma\over 2}}f}_{L^2 }^2.
  \end{eqnarray}
   Observing that the term   $\big|\big(a(v)(-\tilde\triangle_v)^{s}f,~~ p^w f\big)_{L^2}\big|$ on the left hand side is bounded from above by
  \begin{eqnarray*}
    \abs{\biginner{(-\tilde\triangle_v)^{s/2} a(v)(-\tilde\triangle_v)^{s/2}f,~~ p^w f}_{L^2}}+\abs{\biginner{\com{(-\tilde\triangle_v)^{s/2},~a(v)}(-\tilde\triangle_v)^{s/2}f,~~ p^w f}_{L^2}},
  \end{eqnarray*}
  and that
  \[
     \abs{\biginner{\com{(-\tilde\triangle_v)^{s/2},~a(v)}(-\tilde\triangle_v)^{s/2}f,~~ p^w f}_{L^2}}\lesssim  \norm{\comii v^{\gamma/2}(-\tilde\triangle_v)^{s/2}f}_{L^2 }\norm{ \comii v^{s+{\gamma\over 2}}f}_{L^2 }
  \]
  due to  \reff{11060201} and the fact that $p\in S(1, \abs{dv}^2+\abs{d\eta}^2)$,  we have
  \begin{eqnarray*}
    \abs{\biginner{a(v)(-\tilde\triangle_v)^{s}f,~~ p^w f}_{L^2}}\lesssim \norm{ a^{1\over2 }(-\tilde\triangle_v)^{s/2}f}_{L^2}^2+\norm{ a^{1\over2 }(-\tilde\triangle_v)^{s/2} p^wf}_{L^2}^2+\norm{ \comii v^{s+{\gamma\over 2}}f}_{L^2 }^2.
  \end{eqnarray*}
  As for the second term on the right hand side, by virtue of \reff{assumption2} symbolic calculus  (Theorem 2.3.8 and Corollary 2.3.10 of \cite{MR2599384}) shows that
  \[
     \com{a^{1\over2 }(-\tilde\triangle_v)^{s/2},~p^w} \in {\rm Op}\inner{S\inner{\comii v^{s+{\gamma\over2}}+\comii v^{\gamma\over2}\comii \eta^s,~\abs{dv}^2+\abs{d\eta}^2}},
  \]
  and thus
  \begin{eqnarray*}
     \norm{ a^{1\over2 }(-\tilde\triangle_v)^{s/2} p^wf}_{L^2}^2&\lesssim& \norm{ p^w a^{1\over2 }(-\tilde\triangle_v)^{s/2} f}_{L^2}^2+\norm{ \com{a^{1\over2 }(-\tilde\triangle_v)^{s/2},~p^w} f}_{L^2}^2\\
     &\lesssim& \norm{ a^{1\over2 }(-\tilde\triangle_v)^{s/2} f}_{L^2}^2+\norm {\comii{D_\eta}^s \comii v^{\gamma\over2}f}_{L^2}^2+\norm{\comii v^{s+{\gamma\over2}} f}_{L^2}^2.
  \end{eqnarray*}
  Combining the above inequalities, we get \reff{11053110}, completing the proof.
\end{proof}

\subsection{Estimates of the commutators with pseudo-differential operators}\label{sub23}

The main result of this subsection is the  following   estimate on the commutator of $\tilde{\mathcal P}$ with  $M_\eps^s$ which is defined by, with $\eps>0$ and $\xi\in\mathbb R^n$   arbitrary and fixed,
\begin{eqnarray}\label{11052716}
     M_\eps^s =\inner{\varphi_\eps(v,\eta)\comii{\eta}^s}^{w},
\end{eqnarray}
with
\begin{eqnarray}\label{11052701}
  \varphi_\eps(v,\eta)\stackrel{\rm def}{=}\chi\left(\frac{\comii\xi}{\eps \comii{v}^{\gamma}\comii{\eta}^{1+2s}}\right),
\end{eqnarray}
where  $\chi\in C_0^\infty(\mathbb R;~[0,1])$ such that $\chi=1$ in
$[-1,1]$ and supp~$\chi \subset[-2,2]$.

\begin{lemma}\label{110503}
 Let $\tilde{\mathcal P}$ be given in \reff{11051401} with $a$, $b$ satisfying the assumptions \reff{assumption1} and \reff{assumption2}, and let $M_\eps^s$ be defined in \reff{11052716}. Then  for all $f\in\mathcal S(\mathbb R_v^n)$ we have
  \begin{eqnarray}\label{11051110}
     \abs{\biginner{\com{\tilde{\mathcal P},~ M_\eps^{s}} f,~  \comii{v}^{\gamma}M_\eps^{s} f}_{L^2}}
    \lesssim  ~   \eps\norm{\comii v^{\gamma}\comii{D_\eta}^{2s}f}_{L^2}^2 + C_{\eps}\inner{\norm {\comii{D_\eta}^s \comii v^{s+\gamma}f}_{L^2}^2+ \norm{ f}_{L^2}^2}.
  \end{eqnarray}
\end{lemma}

In order to prove the above results  we need some lemmas.

\begin{lemma}\label{lemm110527}
  Let $\varphi_\eps $ and $M_\eps^{s}$ be given in \reff{11052701} and \reff{11052716}. Then $\varphi_\eps \in S\inner{1,~\abs{dv}^2+\abs{d\eta}^2}$ and  $M_\eps^s  \in {\rm Op}\inner{S\inner{\comii{\eta}^s,~\abs{dv}^2+\abs{d\eta}^2}}$, uniformly with respect to $\eps$ and $\xi$. Moreover for any $\alpha$, $\beta\in\mathbb Z_+^n$ there exists a constant $C_{\alpha,\beta}$, depending only on $\alpha$ and $\beta$, such that
  \begin{eqnarray}\label{11053003}
    \abs{\partial_v^\alpha\partial_\eta^\beta  \inner{\varphi_\eps(v,\eta)\comii \eta^s}}\leq   C_{\alpha,\beta} \comii{v}^{-\abs\alpha}\comii \eta^{s-\abs\beta},
  \end{eqnarray}
  and
  \begin{eqnarray}\label{11052702}
    \abs{\partial_v^\alpha\partial_\eta^\beta \inner{\xi\cdot\partial_\eta\varphi_\eps}}\leq \eps~ C_{\alpha,\beta} \comii{v}^{\gamma}\comii\eta^{2s}.
  \end{eqnarray}
\end{lemma}
\begin{proof}
  It is just a straightforward verification, since
\[
   \comii{\xi}  \leq \eps \comii{v}^{\gamma} \comii{\eta}^{ 1+2s}
\]
on the support of $\varphi_\eps$.
The proof is completed.
\end{proof}

  \begin{lemma}\label{lem11050601}
  Let $M_\eps^s$ be given in \reff{11052716}. Then for all $f\in\mathcal S(\mathbb R_v^n)$ we have
  \begin{eqnarray}\label{11051101}
    \abs{\biginner{\com{i \inner{\tau+v\cdot\xi},~ M_\eps^{s}}f,~\comii{v}^{\gamma} M_\eps^{s} f}_{L^2}}
   \lesssim  \eps\norm{\comii v^{\gamma}\comii{D_\eta}^{2s}f}_{L^2}^2.
  \end{eqnarray}
\end{lemma}

\begin{proof}
  Observe
  \[
     \com{i \inner{\tau+v\cdot\xi},~  M_\eps^s}= \frac{ 1}{2\pi} \Big\{\tau+v\cdot\xi,~~\varphi_\eps(v,\eta)\comii{\eta}^s\Big\}^{w}=-\frac{ 1}{2\pi}\Big( \xi\cdot\partial_\eta\inner{\varphi_\eps \comii{\eta}^s}\Big)^{w},
  \]
  where $\big\{\cdot,~\cdot\big\}$ stands for the Poisson bracket defined by
  \begin{eqnarray}\label{11051505}
  \bigset{p,~q}=\frac{\partial p}{\partial\eta}\cdot\frac{\partial q}{\partial v}-\frac{\partial p}{\partial v}\cdot\frac{\partial q}{\partial \eta}.
\end{eqnarray}
Thus
  \begin{eqnarray*}
     \biginner{\com{i \inner{\tau+v\cdot\xi},~  M_\eps^s}f,~\comii{v}^{\gamma} M_\eps^s f}_{L^2}
     = -{1\over {2\pi}}\biginner{ M_\eps^s \comii{v}^{\gamma}\Big( \xi\cdot\partial_\eta\inner{\varphi_\eps\comii{\eta}^s}\Big)^{w} f, ~f}_{L^2}.
  \end{eqnarray*}
  Moreover, in view of \reff{11052702} and \reff{11053010} we have
  \begin{eqnarray*}
    M_\eps^s \comii{v}^{\gamma}\Big( \xi\cdot\partial_\eta\inner{\varphi_\eps\comii{\eta}^s}\Big)^{w}\in{\rm Op}\inner{S\inner{\eps  \comii{v}^{2\gamma}\comii{\eta}^{4s},~\abs{dv}^2+\abs{d\eta}^2}}
  \end{eqnarray*}
  uniformly with respect to $\eps$ and $\xi$.  This implies
\begin{eqnarray*}
     \abs{ \biginner{\com{i \inner{\tau+v\cdot\xi},~  M_\eps^s}f,~\comii{v}^{\gamma} M_\eps^s f}_{L^2}}\lesssim \eps\norm{\comii{v}^{\gamma} \comii{D_v}^{2s}f}_{L^2}^2,
  \end{eqnarray*}
completing the proof of Lemma \ref{lem11050601}.
\end{proof}

The rest of this subsection is occupied by

\begin{proof}[Proof of Lemma \ref{110503}]
  Write
  \[
  \com{\tilde{\mathcal P},~ M_\eps^{s}}=\com{i\inner{t+v\cdot \xi},~ M_\eps^{s}}+a(v)\com{ (-\tilde\triangle_v)^s  ,~ M_\eps^{s}}
  +\com{a ,~ M_\eps^{s}}
  (-\tilde\triangle_v)^s+\com{b,~ M_\eps^{s}}.
  \]
  Then by \reff{11051101} we have
  \begin{eqnarray}\label{11052020}
    \abs{\biginner{\com{\tilde{\mathcal P},~ M_\eps^{s}}f,~\comii{v}^{\gamma} M_\eps^{s} f}_{L^2}}
   \lesssim  \eps\norm{\comii v^{\gamma}\comii{D_\eta}^{2s}f}_{L^2}^2+\sum_{j=1}^3 A_1+A_2+A_3,
  \end{eqnarray}
  with
  \begin{eqnarray*}
      A_1&=& \abs{\biginner{a(v)\com{ (-\tilde\triangle_v)^s  ,~ M_\eps^{s}} f,~\comii{v}^{\gamma} M_\eps^{s} f}_{L^2}},\\
      A_2&=& \abs{\biginner{\com{a ,~ M_\eps^{s}} (-\tilde\triangle_v)^s f,~\comii{v}^{\gamma} M_\eps^{s} f}_{L^2}},\\
      A_3&=& \abs{\biginner{\com{b,~ M_\eps^{s}} f,~\comii{v}^{\gamma} M_\eps^{s} f}_{L^2}}.
  \end{eqnarray*}
  In view of \reff{11053003} we see
  \[
     \com{ (-\tilde\triangle_v)^s  ,~ M_\eps^{s}}\in {\rm Op} \inner{S\inner{\comii{v}^{-1}\comii \eta^{3s-1},~\abs{dv}^2+\abs{d\eta}^2}},
  \]
  and thus
  \[
     a(v)\com{ (-\tilde\triangle_v)^s  ,~ M_\eps^{s}}\in {\rm Op} \inner{S\inner{\comii{v}^{s+\gamma}\comii \eta^{2s},~\abs{dv}^2+\abs{d\eta}^2}}
  \]
  due to  \reff{assumption1} and the fact that $s<1$.  This implies
  \begin{eqnarray*}
      A_1\lesssim  \norm{\comii v^\gamma\comii{D_\eta}^{2s} f}_{L^2}\norm{\comii{D_\eta}^{s}\comii v^{s+\gamma} f}_{L^2}\lesssim
      \eps\norm{\comii v^{\gamma}\comii{D_\eta}^{2s}f}_{L^2}^2+C_\eps \norm{\comii{D_\eta}^{s}\comii v^{s+\gamma} f}_{L^2}^2.
  \end{eqnarray*}
  Similarly, by  \reff{assumption2} and  \reff{11053003}, we conclude  that $ \com{b,~ M_\eps^{s}} \in {\rm Op}
   \inner{S\inner{\comii{v}^{s+\gamma},~\abs{dv}^2+\abs{d\eta}^2}}$ and
  \[
     \com{a ,~ M_\eps^{s}} (-\tilde\triangle_v)^s \in {\rm Op} \inner{S\inner{\comii{v}^{s+\gamma}\comii \eta^{2s},~\abs{dv}^2+\abs{d\eta}^2}},
  \]
  which implies
  \begin{eqnarray*}
      A_2\lesssim  \norm{\comii v^\gamma\comii{D_\eta}^{2s} f}_{L^2}\norm{\comii{D_\eta}^{s}\comii v^{s+\gamma} f}_{L^2}\lesssim
      \eps\norm{\comii v^\gamma\comii{D_\eta}^{2s} f}_{L^2}^2+C_\eps \norm{\comii{D_\eta}^{s}\comii v^{s+\gamma} f}_{L^2}^2,
  \end{eqnarray*}
  and
  \begin{eqnarray*}
      A_3\lesssim  \norm{\comii{D_\eta}^{s}\comii v^{s+\gamma} f}_{L^2}\norm{\comii{v}^{\gamma}f}_{L^2}\lesssim
       \norm{\comii{D_\eta}^{s}\comii v^{s+\gamma} f}_{L^2}^2.
  \end{eqnarray*}
  These inequalities together with \reff{11052020} give  the desired estimate \reff{11051110}, completing the
   proof of Lemma \ref{110503}.
\end{proof}

\section{Proof of  the main results}

In this  section we will proceed to prove Theorem  \ref{th1} by four steps. The first three subsections are devoted to proving the following proposition concerning the hypoellipticity of the operator with parameters, while in the last one we present the proof of Theorem  \ref{th1}.  Since our main analysis is still on  $\mathbb R_v^n$,   we will use the same notation as in the previous section; that is,  $\inner{\cdot, \cdot}_{L^2}$  and  $\norm{\cdot}_{L^2}$ stand for  $\inner{\cdot, \cdot}_{L^2(\mathbb R_v^n)}$  and  $\norm{\cdot}_{L^2(\mathbb R_v^n)}$, respectively.

 \begin{proposition}\label{prp110511}
Let $\tilde{\mathcal P}$ be given in \reff{11051401} with $a$, $b$ satisfying the assumptions \reff{assumption1} and \reff{assumption2}.   Then   for  all $f\in\mathcal S(\mathbb R_v^n)$ we have
  \begin{eqnarray*}
        &&\norm{\comii{v}^{\frac{\gamma-2s}{1+2s}}\comii{\tau}^{\frac{2s}{1+2s}} f}_{L^2 }+\norm{\comii{v}^{\frac{\gamma}{1+2s}}\comii{\xi}^{\frac{2s}{1+2s}} f}_{L^2}+\norm{\comii{v}^{\gamma}\comii{D_v}^{2s}f}_{L^2 }+\norm{\comii{v}^{2s+\gamma}f}_{L^2 }\nonumber\\
        &\lesssim& ~   \norm{\tilde{\mathcal P}f}_{L^2 }+\norm{ f}_{L^2 }.
\end{eqnarray*}
Recall here $\norm{\cdot}_{L^2}$ stands for $\norm{\cdot}_{L^2(\mathbb R^n_v)}$.
\end{proposition}

\subsection{The first part of the proof of  Proposition \ref{prp110511}}

In this subsection we prove the weighted estimate; that is

\begin{lemma}
  Let $\tilde{\mathcal P}$ be given in \reff{11051401} with $a$, $b$ satisfying the assumptions \reff{assumption1} and \reff{assumption2}.  Then   for  all $f\in\mathcal S(\mathbb R_v^n)$ we have
  \begin{eqnarray}\label{11052931}
        \norm {\comii{D_\eta}^s \comii v^{s+\gamma}f}_{L^2}+\norm{\comii{v}^{2s+\gamma}f}_{L^2 }
         \lesssim      \norm{\tilde{\mathcal P}f}_{L^2 }+\norm{ f}_{L^2 }.
\end{eqnarray}
\end{lemma}

\begin{proof}
Let $f\in\mathcal S(\mathbb R_v^n)$. Using \reff{11050715} to the function $\comii{v}^{s+\frac{\gamma}{2}}f$, we have
\begin{eqnarray*}
      &&\norm {\comii{D_\eta}^s \comii v^{s+\gamma}f}_{L^2}^2+\norm{\comii{v}^{2s+\gamma}f}_{L^2}^2 \\
      &\lesssim &\abs{\biginner{\tilde{\mathcal P} \comii{v}^{s+\frac{\gamma}{2}}f,~\comii{v}^{s+\frac{\gamma}{2}}f}_{L^2}}+\norm{\comii{v}^{s+\frac{\gamma}{2}}f}_{L^2}^2\\
      &\lesssim&  \abs{\biginner{\tilde{\mathcal P} f,~\comii{v}^{2s+\gamma}f}_{L^2}}+\abs{\biginner{\com{\tilde{\mathcal P},~ \comii{v}^{s+\frac{\gamma}{2}}} f, ~\comii{v}^{s+\frac{\gamma}{2}}f}_{L^2}}+\norm{\comii{v}^{2s+\gamma}f}_{L^2}\norm{f}_{L^2},
\end{eqnarray*}
which imply that
\begin{eqnarray}\label{11052910}
     \norm {\comii{D_\eta}^s \comii v^{s+\gamma}f}_{L^2}^2+\norm{\comii{v}^{2s+\gamma}f}_{L^2}^2
      \lesssim  \norm{\tilde{\mathcal P} f}_{L^2}^2+\norm{f}_{L^2}^2+\abs{\biginner{\com{\tilde{\mathcal P},  \comii{v}^{s+\frac{\gamma}{2}}} f, \comii{v}^{s+\frac{\gamma}{2}}f}_{L^2}}.
\end{eqnarray}
Moreover, note that  $\com{\tilde{\mathcal P},~ \comii{v}^{s+\frac{\gamma}{2}}}=a(v)\com{(-\tilde\triangle_v)^s,~ \comii{v}^{s+\frac{\gamma}{2}}}$, and thus by \reff{11053010} and \reff{11060401} we have
\[
   \com{\tilde{\mathcal P},~ \comii{v}^{s+\frac{\gamma}{2}}}\in  {\rm Op}\inner{ S\inner{\comii{v}^{3s+\frac{3\gamma}{2}-1}\comii{\eta}^{s},~~\abs{dv}^2+\abs{d\eta}^2}}.
\]
This implies, with $\eps$ sufficiently small,
\begin{eqnarray*}
  \abs{\biginner{\com{\tilde{\mathcal P},~ \comii{v}^{s+\frac{\gamma}{2}}} f, ~\comii{v}^{s+\frac{\gamma}{2}}f}_{L^2}}&\lesssim & \norm {\comii{D_\eta}^s \comii v^{s+\gamma}f}_{L^2}\norm{\comii{v}^{3s+\gamma-1}f}_{L^2}\\& \lesssim& \eps \norm {\comii{D_\eta}^s \comii v^{s+\gamma}f}_{L^2}^2+\eps\norm{\comii{v}^{2s+\gamma}f}_{L^2} +C_\eps \norm{f}_{L^2}^2,
\end{eqnarray*}
where in the last inequality we used the interpolation inequality
\[
   \norm{\comii{v}^{3s+\gamma-1}f}_{L^2}\leq \eps \norm{\comii{v}^{2s+\gamma}f}_{L^2}+C_\eps \norm{f}_{L^2},
\]
due to $s<1$.  Combining \reff{11052910} we get
\begin{eqnarray*}
     &&\norm {\comii{D_\eta}^s \comii v^{s+\gamma}f}_{L^2}^2+\norm{\comii{v}^{2s+\gamma}f}_{L^2}^2\\
      &\lesssim&  \eps \norm {\comii{D_\eta}^s \comii v^{s+\gamma}f}_{L^2}^2+\eps\norm{\comii{v}^{2s+\gamma}f}_{L^2}^2 +C_\eps \norm{f}_{L^2}+ \norm{\tilde{\mathcal P} f}_{L^2}^2.
\end{eqnarray*}
Letting $\eps$ small enough gives the desired estimate \reff{11052931}. The proof is complete.
\end{proof}

\subsection{The second  part of the proof of  Proposition \ref{prp110511}}

The main result in this subsection is the following lemma.

 \begin{lemma}\label{prop11051601}
    Let $\tilde{\mathcal P}$ be given in \reff{11051401} with $a$, $b$ satisfying the assumptions \reff{assumption1} and \reff{assumption2}.  Then   for  all $f\in\mathcal S(\mathbb R_v^n)$ we have
  \begin{eqnarray} \label{11051601}
        \norm{\comii{v}^{\frac{\gamma}{1+2s}}\comii{\xi}^{\frac{2s}{1+2s}} f}_{L^2}\lesssim  \norm{\tilde{\mathcal P} f}_{L^2}   +\norm{f}_{L^2}.
\end{eqnarray}
 \end{lemma}

We would make use of the multiplier method used in \cite{HK2011, Li11} to prove the above result.  Firstly we  need to find a suitable multiplier. In what follows   let $\xi\in\mathbb R^n$ be fixed, and define  a  symbol  $p$   by setting
\begin{eqnarray}\label{11041511}
  p =p_{\xi}(v,\eta)=\frac{\comii{v}^{\gamma/(1+2s)}\xi\cdot\eta}{\comii{\xi}^{2-\frac{2s}{1+2s}}}\psi,
\end{eqnarray}
with $\psi$ given by
\begin{eqnarray}\label{11041512}
  \psi(v,\eta)&=&\chi\left(\frac{\comii{v}^{\gamma}\comii{\eta}^{1+2s}}{\comii{\xi} }\right),
\end{eqnarray}
where  $\chi\in C_0^\infty(\mathbb R;~[0,1])$ such that $\chi=1$ in
$[-1,1]$ and supp $\chi \subset[-2,2]$.

\begin{lemma}\label{lem11051630}
    Let  $p$, $\psi$ be given above. Then  one has $p$, $\psi\in S(1, \abs{dv}^2+\abs{d\eta}^2)$ uniformly with respect to $\xi$.
\end{lemma}

\begin{proof}
  It is just a straightforward verification.
\end{proof}

\begin{lemma}\label{lem042101}
 Let  $\psi$ be given in   \reff{11041512}. Then for all $\abs\alpha+\abs\beta\geq0$ the following inequality
  \begin{eqnarray}\label{11040806}
    \abs{\partial_v^\alpha\partial_\eta^\beta\inner{\xi\cdot\partial_\eta\psi}}\lesssim \comii{v}^{\gamma}\comii\eta^{2s}
  \end{eqnarray}
  holds uniformly with respect to   $\xi$.
\end{lemma}

\begin{proof}
Note that
\[
  \xi\cdot\partial_\eta\psi=\frac{(2s+1)\comii{v}^{\gamma}\comii \eta^{2s-1} \xi\cdot\eta }{\comii{\xi}} \chi'\left(\frac{\comii{v}^{\gamma}\comii{\eta}^{2s+1}}{\comii{\xi}}\right)
\]
Then by  direct computation, \reff{11040806} follows.
The proof of Lemma \ref{lem042101} is thus complete.
\end{proof}

 The rest of this subsection is occupied by

\begin{proof}[Proof of Lemma \ref{prop11051601}]
  Let $f\in \mathcal S(\mathbb R^n_v)$ and let $p ^{w}$ be the Weyl quantization of the symbol $p $ given in \reff{11041511}.  Then using \reff{11051518} gives
  \[
      \abs{\biginner{a(v)(-\tilde\triangle_v)^{s}f+b~ f,~ p^w f}_{L^2}}\lesssim \abs{\biginner{\tilde{\mathcal P} f,~f}_{L^2}}+\norm{f}_{L^2}^2.
  \]
  This together with  the relation
  \begin{eqnarray*}
     {\rm Re}~\biginner{i \inner{\tau+v\cdot\xi} f,~p^w f}_{L^2} = {\rm Re}~\biginner{\tilde{\mathcal P} f,~p^w f}_{L^2}-{\rm Re}~\biginner{a(v)(-\tilde\triangle_v)^{s}f+b~ f,~ p^w f}_{L^2}
  \end{eqnarray*}
  yields
  \begin{eqnarray}\label{11040807}
    {\rm Re}~\biginner{i \inner{\tau+v\cdot\xi} f,~p^w f}_{L^2}\lesssim \abs{\biginner{\tilde{\mathcal P} f,~ f}_{L^2}}+\abs{\biginner{\tilde{\mathcal P} f,~ p^w  f}_{L^2}}+\norm{f}_{L^2}^2.
  \end{eqnarray}
  Next we will give a lower bound of the term on the left side. Observe that
  \begin{eqnarray}\label{11040808}
    {\rm Re}~\biginner{i \inner{\tau+v\cdot\xi} f,~p^w f}_{L^2}=\frac{1}{2\pi}\biginner{\big\{p,~\tau+v\cdot\xi
    \big\}^{w}f,~f}_{L^2},
  \end{eqnarray}
  where $\set{\cdot,~\cdot}$ is the Poisson bracket defined in \reff{11051505}. Direct calculus  shows
  \begin{eqnarray*}
     \big\{p,~\tau+v\cdot\xi\big\}
    &=&\frac{\comii{v}^{\gamma/(1+2s)}\abs\xi^2}{\comii{\xi}^{2-\frac{2s}{1+2s}}}\psi+\frac{\comii{v}^{\gamma/(1+2s)}\xi\cdot\eta}{\comii{\xi} ^{2-\frac{2s}{1+2s}}}\xi\cdot\partial_\eta\psi\\  &=& \comii{v}^{\gamma/(1+2s)}\comii{\xi}^{2s/(1+2s)}\psi-\frac{\comii{v}^{\gamma/(1+2s)}}{\comii{\xi} ^{2-\frac{2s}{1+2s}}}\psi+\frac{\comii{v}^{\gamma/(1+2s)}\xi\cdot\eta}{\comii{\xi} ^{2-\frac{2s}{1+2s}}}\xi\cdot\partial_\eta\psi\\  &=& \comii{v}^{\gamma/(1+2s)}\comii{\xi}^{2s/(1+2s)}-\comii{v}^{\gamma/(1+2s)}\comii{\xi}^{2s/(1+2s)}(1-\psi)\\
    &&-\frac{\comii{v}^{\gamma/(1+2s)}}{\comii{\xi} ^{2-\frac{2s}{1+2s}}}\psi+\frac{\comii{v}^{\gamma/(1+2s)}\xi\cdot\eta}{\comii{\xi} ^{2-\frac{2s}{1+2s}}}\xi\cdot\partial_\eta\psi.
    \end{eqnarray*}
   The above equalities  along with \reff{11040807} and \reff{11040808} yield
 \begin{eqnarray}\label{11052920}
   \norm{ \comii{v}^{\frac{\gamma}{2+4s}}\comii{\xi}^{\frac{s}{1+2s}} f}_{L^2}^2 \lesssim \sum_{j=1}^3 K_j +  \abs{\biginner{\tilde{\mathcal P} f,~ f}_{L^2}}+\abs{\biginner{\tilde{\mathcal P} f,~ p^w  f}_{L^2}}+\norm{f}_{L^2}^2,
\end{eqnarray}
with
\begin{eqnarray*}
  K_1&=&\inner{\inner{\comii{v}^{\gamma/(1+2s)}\comii{\xi}^{2s/(1+2s)}(1-\psi)}^w f,~~f}_{L^2},\\
  K_2&=&\inner{\biginner{\frac{\comii{v}^{\gamma/(1+2s)}}{\comii{\xi} ^{2-\frac{2s}{1+2s}}}\psi}^w f,~~f}_{L^2},\\
  K_3&=&-\inner{\biginner{\frac{\comii{v}^{\gamma/(1+2s)}\xi\cdot\eta}{\comii{\xi} ^{2-\frac{2s}{1+2s}}}\xi\cdot\partial_\eta\psi}^w f,~~f}_{L^2}.
\end{eqnarray*}
Note that
 \[
     \comii{\xi}^{2s/(1+2s)} \leq  \comii{v}^{2s\gamma/(1+2s)}\comii{\eta}^{2s}
    \]
on the support of $\partial_v^\alpha\partial_\eta^\beta(1-\psi)$ with $\abs\alpha+\abs\beta\geq 0$. Then by virtue of the  conclusion that $\psi\in S(1, \abs{dv}^2+\abs{d\eta}^2)$ uniformly with respect to $\xi$ in Lemma \ref{lem11051630}, we have
\[
   \comii{v}^{\gamma/(1+2s)}\comii{\xi}^{2s/(1+2s)}(1-\psi) \in S\inner{\comii{v}^{\gamma}\comii{\eta}^{2s}, ~\abs{dv}^2+\abs{d\eta}^2}
\]
uniformly with respect to $\xi$ . This implies
\[
   \comii{D_\eta}^{-s}\comii{v}^{-\gamma/2}\inner{\comii{v}^{\gamma/(1+2s)}\comii{\xi}^{2s/(1+2s)}(1-\psi)}^w \comii{v}^{-\gamma/2}\comii{D_\eta}^{-s}
   \in {\rm Op}\inner{S\inner{1, ~~\abs{dv}^2+\abs{d\eta}^2}},
\]
and thus
\begin{eqnarray}\label{11052921}
   K_1\lesssim \norm{\comii{D_\eta}^s\comii v^{\gamma/2}f}_{L^2}^2  \lesssim \abs{\biginner{\tilde{\mathcal P} f,~ f}_{L^2}}+\norm{f}_{L^2}^2,
\end{eqnarray}
where the last inequality follows from \reff{11050715}.
Furthermore since
\[
   \frac{\comii{v}^{\gamma/(1+2s)}}{\comii{\xi} ^{2-\frac{2s}{1+2s}}} \leq \frac{\comii{\xi}^{1/(1+2s)}}{\comii{\xi} ^{2-\frac{2s}{1+2s}}}\leq 1
\]
on the support of $ \psi$, then combining the fact that that $\psi\in S(1, \abs{dv}^2+\abs{d\eta}^2)$ uniformly with respect to $\xi$ we conclude
\[
    \frac{\comii{v}^{\gamma/(1+2s)}}{\comii{\xi} ^{2-\frac{2s}{1+2s}}}\psi \in S\inner{1, ~~\abs{dv}^2+\abs{d\eta}^2},
\]
which implies
\begin{eqnarray}\label{11052922}
   K_2\lesssim \norm{f}_{L^2}^2.
\end{eqnarray}
It remains to treat $K_3$. Direct verification shows
\[
   \abs{\partial_v^\alpha\partial_\eta^\beta \biginner{\frac{\comii{v}^{\gamma/(1+2s)}\xi\cdot\eta}{\comii{\xi} ^{2-\frac{2s}{1+2s}}}}}\leq 1
\]
on the support of $ \psi$.  This along with \reff{11040806} gives
\[
   \frac{\comii{v}^{\gamma/(1+2s)}\xi\cdot\eta}{\comii{\xi} ^{2-\frac{2s}{1+2s}}}\xi\cdot\partial_\eta\psi\in  S\inner{\comii{v}^{\gamma}\comii{\eta}^{2s}, ~~\abs{dv}^2+\abs{d\eta}^2}.
\]
As a result, repeating the arguments used in the treatment of $K_1$ yields
\[
   K_3\lesssim \norm{\comii{D_\eta}^s\comii v^{\gamma/2}f}_{L^2}^2  \lesssim \abs{\biginner{\tilde{\mathcal P} f,~ f}_{L^2}}+\norm{f}_{L^2}^2.
\]
This, together with \reff{11052920},   \reff{11052921} and \reff{11052922}, gives
\begin{eqnarray*}\label{11041105}
   \norm{ \comii{v}^{\frac{\gamma}{2+4s}}\comii{\xi}^{\frac{s}{1+2s}} f}_{L^2}^2 \lesssim  \abs{\biginner{\tilde{\mathcal P} f,~ f}_{L^2}}+\abs{\biginner{\tilde{\mathcal P} f,~ p^w  f}_{L^2}}+\norm{f}_{L^2}^2.
\end{eqnarray*}
Now applying the above inequality to the function $\comii v^{\frac{\gamma}{2+4s}}f$, we get
\begin{eqnarray*}
  &&\norm{\comii{v}^{\frac{\gamma}{1+2s}}\comii{\xi}^{\frac{s}{1+2s}}f}_{L^2}^2\\
   &\lesssim & \abs{\biginner{\tilde{\mathcal P} \comii v^{\frac{\gamma}{2+4s}}f,~ \comii v^{\frac{\gamma}{2+4s}}f}_{L^2}}+\abs{\biginner{\tilde{\mathcal P} \comii v^{\frac{\gamma}{2+4s}}f,~ p^w \comii v^{\frac{\gamma}{2+4s}}f}_{L^2}}+\norm{\comii v^{\frac{\gamma}{2+4s}}f}_{L^2}^2\\
   &\lesssim& \abs{\biginner{\tilde{\mathcal P} f,~\comii v^{\frac{\gamma}{1+2s}} f}_{L^2}}+\abs{\biginner{\tilde{\mathcal P} f,~\comii v^{\frac{\gamma}{2+4s}}p^w\comii v^{\frac{\gamma}{2+4s}}f}_{L^2}}+\norm{\comii v^{\frac{\gamma}{1+2s}}f}_{L^2}\norm{f}_{L^2}\\
   &&+\abs{\biginner{\com{\tilde{\mathcal P},~\comii v^{\frac{\gamma}{2+4s}}} f,~ \comii v^{\frac{\gamma}{2+4s}}f}_{L^2}}+\abs{\biginner{\com{\tilde{\mathcal P},~\comii v^{\frac{\gamma}{2+4s}}} f,~p^w \comii v^{\frac{\gamma}{2+4s}}f}_{L^2}}\\
   &\lesssim&   \norm{\tilde{\mathcal P} f}_{L^2}\norm{\comii{v}^{\frac{\gamma}{1+2s}} f}_{L^2}+\norm{\comii v^{\frac{\gamma}{1+2s}}f}_{L^2}\norm{f}_{L^2}\\
   &&+\abs{\biginner{\com{\tilde{\mathcal P},~\comii v^{\frac{\gamma}{2+4s}}} f,~ \comii v^{\frac{\gamma}{2+4s}}f}_{L^2}}+\abs{\biginner{\com{\tilde{\mathcal P},~\comii v^{\frac{\gamma}{2+4s}}} f,~p^w \comii v^{\frac{\gamma}{2+4s}}f}_{L^2}}.
\end{eqnarray*}
On the other hand,  by   \reff{11053010} and \reff{11060401}  we have,
\[
   \com{\tilde{\mathcal P},~\comii v^{\frac{\gamma}{2+4s}}}= a(v)\com{(-\tilde\triangle_v)^{s},~\comii v^{\frac{\gamma}{2+4s}}}\in {\rm Op}\inner{ S\inner{\comii{v}^{s+\gamma+\frac{\gamma}{2+4s}}\comii{\eta}^{s}, \abs{dv}^2+\abs{d\eta}^2}}.
\]
Then symbolic calculus gives
\begin{eqnarray*}
  &&\abs{\biginner{\com{\tilde{\mathcal P},~\comii v^{\frac{\gamma}{2+4s}}} f,~ \comii v^{\frac{\gamma}{2+4s}}f}_{L^2}}+\abs{\biginner{\com{\tilde{\mathcal P},~\comii v^{\frac{\gamma}{2+4s}}} f,~p^w \comii v^{\frac{\gamma}{2+4s}}f}_{L^2}}\\
  &\lesssim & \norm {\comii{D_\eta}^s \comii v^{s+\gamma}f}_{L^2}\norm{\comii{v}^{\frac{\gamma}{1+2s}}f}_{L^2},
\end{eqnarray*}
since $p\in S\inner{1, \abs{dv}^2+\abs{d\eta}^2}$ uniformly with respect to $\xi$.  Consequently combining the above inequalities, we have
\begin{eqnarray*}
    \norm{\comii{v}^{\frac{\gamma}{1+2s}}\comii{\xi}^{\frac{s}{1+2s}}f}_{L^2}^2 \lesssim  \norm{\comii{v}^{\frac{\gamma}{1+2s}} f}_{L^2} \biginner{ \norm{\tilde{\mathcal P} f}_{L^2}+\norm{ f}_{L^2}+ \norm {\comii{D_\eta}^s \comii v^{s+\gamma}f}_{L^2}}.
\end{eqnarray*}
Note that $\norm{\cdot}_{L^2}$ stands for the norm in $L^2(\mathbb R_v^n)$. Then multiplying both sides the factor $\comii{\xi}^{2s/(1+2s)}$,  we get
\begin{eqnarray*}
   \norm{\comii{v}^{\frac{\gamma}{1+2s}}\comii{\xi}^{\frac{2s}{1+2s}}f}_{L^2}^2
     \lesssim    \norm{\comii{v}^{\frac{\gamma}{1+2s}} \comii{\xi}^{\frac{2s}{1+2s}}f}_{L^2}\biginner{ \norm{\tilde{\mathcal P} f}_{L^2}+\norm{ f}_{L^2}+ \norm {\comii{D_\eta}^s \comii v^{s+\gamma}f}_{L^2}},
\end{eqnarray*}
and thus
\begin{eqnarray*}
     \norm{\comii{v}^{\frac{\gamma}{1+2s}}\comii{\xi}^{\frac{2s}{1+2s}}f}_{L^2}
    \lesssim  \norm{\tilde{\mathcal P} f}_{L^2}   +\norm{ f}_{L^2}+ \norm {\comii{D_\eta}^s \comii v^{s+\gamma}f}_{L^2}\lesssim  \norm{\tilde{\mathcal P} f}_{L^2}   +\norm{ f}_{L^2},
\end{eqnarray*}
where the last inequality follows from \reff{11052931}. This gives the desired estimate \reff{11051601}, completing the proof of Lemma \ref{prop11051601}.
\end{proof}

\subsection{End of the proof of  Proposition \ref{prp110511}}

In view of \reff{11052931} and \reff{11051601}, the proof  of  Proposition \ref{prp110511} will be complete if we could show the following
lemma.

\begin{lemma}\label{lem11060316}
  Let $\tilde{\mathcal P}$ be given in \reff{11051401} with $a$, $b$ satisfying the assumptions \reff{assumption1} and \reff{assumption2}.   Then   for  all $f\in\mathcal S(\mathbb R_v^n)$ we have
  \begin{eqnarray}\label{11053050}
        \norm{\comii{v}^{\frac{\gamma-2s}{1+2s}}\comii{\tau}^{\frac{2s}{1+2s}} f}_{L^2 }+\norm{\comii{v}^{\gamma}\comii{D_v}^{2s}f}_{L^2 } \lesssim      \norm{\tilde{\mathcal P}f}_{L^2 }+\norm{ f}_{L^2 }.
  \end{eqnarray}
\end{lemma}

\begin{proof}
Let $f\in\mathcal S(\mathbb R_v^n)$. We first treat the second term on the left hand side of \reff{11053050}.   By  \reff{11060401} one has
\begin{eqnarray*}
    \norm{\comii v^{\gamma}\comii{D_\eta}^{2s}f}_{L^2}^2
     &\lesssim & \norm{\comii{D_\eta}^{s} \comii v^{\gamma}\comii{D_\eta}^{s}f}_{L^2}^2 +\norm{\com{\comii{D_\eta}^{s},~\comii v^{\gamma}}\comii{D_\eta}^{s}f}_{L^2}^2\\
       &\lesssim & \norm{\comii{D_\eta}^{s}\comii v^{\gamma} \big(\comii{\eta}^{s}\big)^wf}_{L^2}^2 +\norm{ \comii v^{\gamma} \comii{D_\eta}^{s}f}_{L^2}^2.
\end{eqnarray*}
 Moreover for the last  term in the above inequality we have
\begin{eqnarray*}
      \norm{ \comii v^{\gamma} \comii{D_\eta}^{s}f}_{L^2}^2 &\lesssim  & \norm{ \comii v^{s+\gamma} \comii{D_\eta}^{s}f}_{L^2}^2
      \lesssim   \norm{\comii{D_\eta}^{s} \comii{v}^{s+\gamma} f}_{L^2}^2+\norm{\com{\comii{D_\eta}^{s},~ \comii{v}^{s+\gamma}} f}_{L^2}^2\\
     &\lesssim  & \norm{\comii{D_\eta}^{s} \comii{v}^{s+\gamma} f}_{L^2}^2+\norm{ \comii{v}^{2s+\gamma}f}_{L^2}^2\\
     &\lesssim  & \norm{\tilde{\mathcal P}f}_{L^2 }^2+\norm{ f}_{L^2 }^2,
\end{eqnarray*}
 the last inequality using \reff{11052931}.  As a result the desired upper bound for $\norm{\comii v^{\gamma} \comii{D_\eta}^{2s}f}_{L^2}$ will follow if  we could prove that, with  $\eps>0$ arbitrarily small,
\begin{eqnarray}\label{11053116}
  \norm{\comii{D_\eta}^{s}\comii v^{\gamma} \big(\comii{\eta}^{s}\big)^wf}_{L^2}^2\lesssim \eps\norm{\comii v^{\gamma} \comii{D_\eta}^{2s}f}_{L^2}^2+C_\eps\inner{\norm{\tilde{\mathcal P}f}_{L^2 }^2+\norm{ f}_{L^2 }^2}.
\end{eqnarray}
In order to show the above inequality  we
write
\begin{eqnarray}\label{11060101}
  \norm{\comii{D_\eta}^{s}\comii v^{\gamma} \big(\comii{\eta}^{s}\big)^wf}_{L^2}^2\lesssim J_1+J_2,
\end{eqnarray}
with
\begin{eqnarray*}
  J_1&=&\norm{\comii{D_\eta}^{s}\comii v^{\gamma} \big(\varphi_\eps \comii{\eta}^{s}\big)^wf}_{L^2}^2=\norm{\comii{D_\eta}^{s}\comii v^{\gamma} M_\eps^s f}_{L^2}^2,\\
  J_2&=&\norm{\comii{D_\eta}^{s}\comii v^{\gamma} \big((1-\varphi_\eps)\comii{\eta}^{s}\big)^wf}_{L^2}^2,
\end{eqnarray*}
where $M_\eps^s$ and $\varphi_\eps$ are defined in \reff{11052716} and \reff{11052701}.  Let's first  treat the term $J_2$.  Writing
 \begin{eqnarray*}
   J_2&=&\biginner{\comii{D_\eta}^{2s}\comii v^{\gamma}f,~~\, \big((1-\varphi_\eps)\comii{\eta}^{s}\big)^w\comii v^{\gamma} \big((1-\varphi_\eps)\comii{\eta}^{s}\big)^wf}_{L^2}\\
   &&+\biginner{\com{\comii{D_\eta}^{2s},~~\,\big((1-\varphi_\eps)\comii{\eta}^{s}\big)^w}\comii v^{\gamma}f,~~ \comii v^{\gamma} \big((1-\varphi_\eps)\comii{\eta}^{s}\big)^wf}_{L^2},
 \end{eqnarray*}
 we have by direct symbolic calculus
   \begin{eqnarray*}
   J_2&\leq&\norm{\comii v^{\gamma}\comii{D_\eta}^{2s}f}_{L^2} \norm{\big((1-\varphi_\eps)\comii{\eta}^{s}\big)^w\comii v^{\gamma} \big((1-\varphi_\eps)\comii{\eta}^{s}\big)^wf}_{L^2} \\
   &&+\norm{\comii v^{\gamma}\comii{D_\eta}^{2s}f}_{L^2} \norm{\comii v^{\gamma} \big((1-\varphi_\eps)\comii{\eta}^{s}\big)^wf}_{L^2}.
 \end{eqnarray*}
   Moreover observe that the symbols of the operators
   \[
    \big((1-\varphi_\eps)\comii{\eta}^{s}\big)^w\comii v^{\gamma} \big((1-\varphi_\eps)\comii{\eta}^{s}\big)^w ~~\,{\rm and}\,~~  \comii v^{\gamma} \big((1-\varphi_\eps)\comii{\eta}^{s}\big)^w
   \]
    belong  to
\begin{eqnarray*}
      S\inner{\eps^{-\frac{2s}{1+2s}} \comii{v}^{\frac{\gamma}{1+2s}}\comii{\xi}^{\frac{2s}{1+2s}},~\abs{dv}^2+\abs{d\eta}^2}
\end{eqnarray*}
uniformly with respect to $\eps$ and $\xi$, due to the fact that
\begin{eqnarray*}
   \comii{v}^\gamma\comii{\eta}^{1+2s}\leq \eps^{-1}\comii \xi
\end{eqnarray*}
on the support of $\partial_{v}^\alpha\partial^\beta_\eta(1-\varphi_\eps)$ with $\abs\alpha+\abs\beta\geq0$. Then
 \begin{eqnarray}\label{11053031}
 \begin{split}
   J_2&\lesssim
    \eps^{-\frac{2s}{1+2s}} \norm{\comii v^{\gamma}\comii{D_\eta}^{2s}f}_{L^2} \norm{\comii{v}^{\frac{\gamma}{1+2s}}\comii{\xi}^{\frac{2s}{1+2s}}f}_{L^2}^2\\
    &\lesssim \eps  \norm{\comii v^{\gamma}\comii{D_\eta}^{2s}f}_{L^2}^2+ C_\eps\inner{\norm{\tilde{\mathcal P}f}_{L^2 }^2+\norm{ f}_{L^2 }^2},
 \end{split}
 \end{eqnarray}
 the last inequality using \reff{11051601}.
 Next we  treat $J_1$.  Applying \reff{11050715} to the function $\comii{v}^{\gamma/2}M_\eps^s f$  gives
  \begin{eqnarray*}
    J_1
    &\lesssim& \abs{\biginner{\tilde{\mathcal P} \comii{v}^{\gamma/2}M_\eps^s f,~ \comii{v}^{\gamma/2}M_\eps^s f }_{L^2}}+\norm{\comii{v}^{\gamma/2}M_\eps^s f}_{L^2}^2 \\
    &\lesssim&  J_{1,1}+J_{1,2}+J_{1,3},
  \end{eqnarray*}
  with
  \begin{eqnarray*}
    J_{1,1}&=& \abs{\biginner{\com{\tilde{\mathcal P},~\comii{v}^{\gamma/2} } M_\eps^sf, ~\comii{v}^{\gamma/2}M_\eps^s f }_{L^2}},\\
    J_{1,2}&=& \abs{\biginner{\com{\tilde{\mathcal P},~M_\eps^s} f,~ \comii{v}^{\gamma}M_\eps^s f }_{L^2}},\\
    J_{1,3}&=& \norm{\tilde{\mathcal P} f}_{L^2} \norm{M_\eps^s\comii{v}^{\gamma}M_\eps^s f}_{L^2}+\norm{\comii{v}^{\gamma/2}M_\eps^s f}_{L^2}^2.
  \end{eqnarray*}
  Next we will proceed to handle the above three terms. It's clear that
  \[
     J_{1,3} \leq  \eps\norm{\comii v^{\gamma}\comii{D_\eta}^{2s}f}_{L^2}^2 + C_{\eps} \inner{ \norm{\tilde{\mathcal P}f}_{L^2}^2+\norm{f}_{L^2}^2},
  \]
  since
  \[
      M_\eps^s\comii{v}^{\gamma}M_\eps^s \comii{D_\eta}^{-2s}\comii v^{-\gamma} \in S\inner{1,~\abs{dv}^2+\abs{d\eta}^2}
  \]
  uniformly with respect to $\eps$ and $\xi$, by virtue of the conclusions in Lemma \ref{lemm110527}.
  Using \reff{11051110} in Lemma \ref{110503}  gives
   \begin{eqnarray*}
    J_{1,2} &\lesssim &  \eps\norm{\comii v^{\gamma}\comii{D_\eta}^{2s}f}_{L^2}^2 + C_{\eps}\inner{\norm {\comii{D_\eta}^s \comii v^{s+\gamma}f}_{L^2}^2+\norm{f}_{L^2}^2}\\
    &\lesssim &  \eps\norm{\comii v^{\gamma}\comii{D_\eta}^{2s}f}_{L^2}^2 + C_{\eps}\inner{\norm{\tilde{\mathcal P}f}_{L^2}^2+\norm{f}_{L^2}^2},
  \end{eqnarray*}
  the last inequality following from \reff{11052931}. Finally as for the term $J_{1,1}$,  by \reff{11053010} and \reff{11060401}   we have
  \[
      \com{\tilde{\mathcal P},~\comii{v}^{\gamma/2} }=a(v)\com{(-\tilde\triangle_v)^{s} ,~\comii{v}^{\gamma/2}}\in {\rm Op}\inner{S\inner{\comii v^{s+3\gamma/2}\comii\eta^s,~\abs{dv}^2+\abs{d\eta}^2}},
  \]
  due to $0<s<1$. This implies
  \begin{eqnarray*}
    J_{1,1}&=& \abs{\biginner{\com{\tilde{\mathcal P},~\comii{v}^{\gamma/2} } M_\eps^sf, ~\comii{v}^{\gamma/2}M_\eps^s f }_{L^2}} \\
    &\lesssim & \eps\norm{\comii v^{\gamma}\comii{D_\eta}^{2s}f}_{L^2}^2+C_\eps \norm{\comii{D_\eta}^s\comii{v}^{s+\gamma} f}_{L^2}^2\\
    &\lesssim & \eps\norm{\comii v^{\gamma}\comii{D_\eta}^{2s}f}_{L^2}^2+C_\eps \inner{\norm{\tilde{\mathcal P} f}_{L^2}^2+\norm{ f}_{L^2}^2},
  \end{eqnarray*}
  the last inequality following from \reff{11052931}.
  This along with the estimates on the terms $J_{1,2}$ and $J_{1,3}$  gives
  \begin{eqnarray*}
    J_{1}\lesssim J_{1,1}+J_{1,2}+J_{1,3}  \lesssim  \eps\norm{\comii v^{\gamma}\comii{D_\eta}^{2s}f}_{L^2}^2+C_\eps \inner{\norm{\tilde{\mathcal P} f}_{L^2}^2+\norm{ f}_{L^2}^2}.
  \end{eqnarray*}
  Then the desired estimate \reff{11053116} follows from the combination of  \reff{11060101}, \reff{11053031} and  the above inequality, giving  the upper bound for the  second term on the left hand side of \reff{11053050}; that is
\begin{eqnarray}\label{11060105}
  \norm{\comii v^{\gamma} \comii{D_\eta}^{2s}f}_{L^2}\lesssim \norm{\tilde{\mathcal P}f}_{L^2 } +\norm{ f}_{L^2 } .
\end{eqnarray}
  Now it remains to treat the first term. By computation, we have
 \begin{eqnarray*}
   \comii v^{\frac{\gamma-2s}{1+2s}}\comii\tau^{\frac{2s}{1+2s}} &\lesssim&  \comii v^{\frac{\gamma-2s}{1+2s}}\comii{\tau+v\cdot\xi}^{\frac{2s}{1+2s}} +\comii v^{\frac{\gamma-2s}{1+2s}}\comii{v\cdot\xi}^{\frac{2s}{1+2s}}\\
   &\lesssim & \comii v^{\frac{\gamma+2s}{1+2s}}\biginner{\comii v^{\frac{-4s}{1+2s}}\comii{\tau+v\cdot\xi}^{\frac{2s}{1+2s}}} +\comii v^{\frac{\gamma-2s}{1+2s}}\comii{v}^{\frac{2s}{1+2s}}\comii{\xi}^{\frac{2s}{1+2s}}\\
   &\lesssim & \comii v^{2s+\gamma}+ \comii v^{{-2}}\inner{\abs{\tau+v\cdot\xi}+1}  +\comii v^{\frac{\gamma}{1+2s}}\comii{\xi}^{\frac{2s}{1+2s}},
 \end{eqnarray*}
 where the last inequality follows from  the Young's inequality
 \[
     \comii v^{\frac{\gamma+2s}{1+2s}}\biginner{\comii v^{\frac{-4s}{1+2s}}\comii{\tau+v\cdot\xi}^{\frac{2s}{1+2s}}}  \leq \frac{\inner{\comii v^{\frac{\gamma+2s}{1+2s}}}^{1+2s}}{1+2s}+\frac{2s}{1+2s}\biginner{\comii v^{\frac{-4s}{1+2s}}\comii{\tau+v\cdot\xi}^{\frac{2s}{1+2s}}}^{(1+2s)/(2s)}.
 \]
 As a result, using the relation $i(\tau+v\cdot\xi) f=\tilde{\mathcal P}f-a(v)(-\tilde\triangle _v)^sf-b(v)f$, we compute
 \begin{eqnarray*}
   &&\norm{\comii v^{\frac{\gamma-2s}{1+2s}}\comii\tau^{\frac{2s}{1+2s}}f}_{L^2}\\
   & \lesssim& \norm{\comii v^{{-2}}\inner{\tau+v\cdot\xi} f}_{L^2} +
   \norm{ \comii v^{2s+\gamma} f}_{L^2}+\norm{\comii v^{-2} f}_{L^2}
    +\norm{ \comii v^{\frac{\gamma}{1+2s}}\comii{\xi}^{\frac{2s}{1+2s}} f}_{L^2}\\
   & \lesssim& \norm{\comii v^{{-2}}\tilde{\mathcal P} f}_{L^2} +\norm{\comii v^{{-2}} a(v)(-\tilde\triangle_v)^{s} f}_{L^2}+\norm{\comii v^{{-2}}b(v) f}_{L^2}\\
   & &  +\norm{ \comii v^{2s+\gamma} f}_{L^2}+\norm{\comii v^{-2} f}_{L^2}+\norm{ \comii v^{\frac{\gamma}{1+2s}}\comii{\xi}^{\frac{2s}{1+2s}} f}_{L^2}\\
   & \lesssim& \norm{\tilde{\mathcal P} f}_{L^2}+\norm{ f}_{L^2} +\norm{\comii v^{\gamma} (-\tilde\triangle_v)^{s} f}_{L^2}
    +\norm{ \comii v^{2s+\gamma} f}_{L^2}+\norm{ \comii v^{\frac{\gamma}{1+2s}}\comii{\xi}^{\frac{2s}{1+2s}} f}_{L^2},
\end{eqnarray*}
where the last inequality follows from  \reff{assumption1} and \reff{assumption2} . Then using  \reff{11052931},  \reff{11051601} and \reff{11060105} to control the last three terms,  we get
\begin{eqnarray*}
    \norm{\comii v^{\frac{\gamma-2s}{1+2s}}\comii\tau^{\frac{2s}{1+2s}}f}_{L^2}
     \lesssim  \norm{\tilde{\mathcal P} f}_{L^2} + \norm{ f}_{L^2},
\end{eqnarray*}
completing the proof of Lemma \ref{lem11060316}.
\end{proof}

\subsection{Proof of  Theorem  \ref{th1}}

Now we are ready to prove  Theorem  \ref{th1}, which can be deduced at once from the following lemma  by taking the partial Fourier  transform with respect to $t,~x$ variables.

\begin{lemma}\label{lem110603}
  Given $m\in \mathbb R$,  there exist a constant $C_m$ depending only on $m$,  such that for all $\tau\in\mathbb R$ and all $\xi\in\mathbb R^n$, and all $f\in\mathcal S\inner{\mathbb R_v^n}$ we have
  \begin{eqnarray*}
  \begin{split}
         &\norm{\Lambda^m\comii{v}^{\frac{\gamma-2s}{1+2s}}\comii{\tau}^{\frac{2s}{1+2s}} f}_{L^2 }+\norm{\Lambda^m \comii{v}^{\frac{\gamma}{1+2s}}\comii{\xi}^{\frac{2s}{1+2s}} f}_{L^2}+\norm{\Lambda^m \comii{v}^{\gamma}\comii{D_v}^{2s}f}_{L^2 }+\norm{\Lambda^m \comii{v}^{2s+\gamma}f}_{L^2 } \\
         &\lesssim ~  C_m\inner{ \norm{\Lambda^m \tilde{\mathcal P}f}_{L^2 }+\norm{ \Lambda^m f}_{L^2 }},
\end{split}
\end{eqnarray*}
where  $\norm{\cdot}_{L^2}$ stands for $\norm{\cdot}_{L^2(\mathbb R_v^n)}$, and
$
   \Lambda^m= \inner{1+\abs{\tau}^2+\abs\xi^2+\abs{D_\eta}^2}^{m\over2}.
$
\end{lemma}

\begin{proof}
For any $\tau\in\mathbb R$ and any $\xi\in\mathbb R^n$, we denote
\[\lambda(\eta)=\lambda_{\tau,\xi}(\eta)=\inner{1+\abs{\tau}^2+\abs\xi^2+\abs{\eta}^2}^{1\over2}.\]
Then by direct verification we see  $ \Lambda^m\in {\rm Op}\inner{S\inner{\lambda^m, \abs{dv}^2+\frac{\abs{d\eta}^2}{\lambda^2}}}$ uniformly with respect to $\tau$ and $\xi$. Then symbolic calculus  (Theorem 2.3.8 and Corollary 2.3.10 of \cite{MR2599384}) shows that
\begin{eqnarray}\label{11060205}
  \forall ~\ell\in\mathbb R,\quad \com{\Lambda^m,~\comii v^\ell}\in {\rm Op}\inner{S\inner{\comii v^{\ell-1}\lambda^{m-1}, \abs{dv}^2+\abs{d\eta}^2}}
\end{eqnarray}
and that
\begin{eqnarray}\label{11060206}
   \com{\Lambda^m,~a},~\,\,\com{\Lambda^m,~b}\in {\rm Op}\inner{S\inner{\comii v^{2s+\gamma-1}\lambda^{m-1}, \abs{dv}^2+\abs{d\eta}^2}},
\end{eqnarray}
uniformly with respect to $\tau$ and $\xi$.  As a result,  combining  \reff{11053010}, \reff{11060206} and  the fact that $s<1$,  we have
\[
    \com{\Lambda^m,~a}(-\tilde\triangle_v)^s \comii v^{-(s+\gamma)} \comii{D_\eta}^{-s}\Lambda^{-(m-1+s)}  \in {\rm Op} \inner{S\inner{1, \abs{dv}^2+\abs{d\eta}^2}}.
\]
This along with the relation
\[
    \com{\Lambda^m,~a}(-\tilde\triangle_v)^s =\inner{\com{\Lambda^m,~a}(-\tilde\triangle_v)^s \comii v^{-(s+\gamma)} \comii{D_\eta}^{-s}\Lambda^{-(m-1+s)}}\Lambda^{m-1+s}\comii{D_\eta}^s\comii v^{s+\gamma},
\]
 implies
\begin{eqnarray}\label{11060310}
\begin{split}
  \norm{\com{\Lambda^m,~a}(-\tilde\triangle_v)^s f}_{L^2}&\lesssim  \norm{\Lambda^{m-1+s}\comii{D_\eta}^s \comii v^{s+\gamma }  f}_{L^2}\\&\lesssim
  \eps \norm{\Lambda^{m}\comii{D_\eta}^s\comii v^{s+\gamma } f}_{L^2}+C_\eps \norm{\Lambda^{m}\comii v^{s+\gamma } f}_{L^2},
\end{split}
\end{eqnarray}
the last inequality using the interpolation inequality that, with $\eps$ arbitrarily small,
\[
\norm{\Lambda^{m-1+s}\comii{D_\eta}^s\comii v^{s+\gamma } f}_{L^2}\lesssim \eps \norm{\Lambda^{m}\comii{D_\eta}^s\comii v^{s+\gamma } f}_{L^2}+C_\eps \norm{\Lambda^{m}\comii v^{s+\gamma } f}_{L^2}.
\]
By  virtue of \reff{11053010} and  \reff{11060205}, using the similar arguments as above  we could prove that
\begin{eqnarray*}
    \norm{\comii{D_\eta}^s\com{\Lambda^{m},~\comii v^{s+\gamma }} f}_{L^2}+  \norm{\com{\Lambda^{m},~\comii v^{s+\gamma }} f}_{L^2}.
    \lesssim  \norm{\comii v^{s+\gamma } \Lambda^{m} f}_{L^2}.
\end{eqnarray*}
As a result,
\begin{eqnarray*}
   &&  \norm{\Lambda^{m}\comii{D_\eta}^s\comii v^{s+\gamma } f}_{L^2}+   \norm{\Lambda^{m}\comii v^{s+\gamma } f}_{L^2} \\
   &\lesssim&   \norm{\comii{D_\eta}^s\comii v^{s+\gamma }  \Lambda^{m} f}_{L^2}+   \norm{\comii v^{s+\gamma }\Lambda^{m} f}_{L^2}\\
   &&+  \norm{\comii{D_\eta}^s\com{\Lambda^{m},~\comii v^{s+\gamma }} f}_{L^2}+    \norm{\com{\Lambda^{m},~\comii v^{s+\gamma }} f}_{L^2}\\
   &\lesssim&   \norm{\comii{D_\eta}^s\comii v^{s+\gamma }  \Lambda^{m} f}_{L^2}+   \norm{\comii v^{s+\gamma }\Lambda^{m} f}_{L^2}\\
   &\lesssim&   \norm{\comii{D_\eta}^s\comii v^{s+\gamma }  \Lambda^{m} f}_{L^2}+\eps'\norm{\comii v^{2s+\gamma }\Lambda^{m} f}_{L^2}+  C_{\eps'} \norm{\Lambda^{m} f}_{L^2},
\end{eqnarray*}
where   $\eps'$ are arbitrarily small,  and  the last inequality follows from  the interpolation inequality
\[
  \norm{\comii v^{s+\gamma }\Lambda^{m} f}_{L^2} \lesssim \eps'\norm{\comii v^{2s+\gamma }\Lambda^{m} f}_{L^2}+  C_{\eps'} \norm{\Lambda^{m} f}_{L^2}.
\]
The above inequalities along with \reff{11060310} yield
\begin{eqnarray*}
  \norm{\com{\Lambda^m,~a}(-\tilde\triangle_v)^s f}_{L^2} \lesssim
  \eps \norm{\comii{D_\eta}^s\comii v^{s+\gamma }  \Lambda^{m} f}_{L^2} +\eps\norm{\comii v^{2s+\gamma }\Lambda^{m} f}_{L^2}+C_\eps \norm{\Lambda^{m}f}_{L^2}.
\end{eqnarray*}
Moreover by direct calculus we could verify that
  \[
      \norm{ \comii{D_\eta}^s  \comii v^{s+\gamma}  \Lambda^{m} f}_{L^2}\lesssim  \norm{\comii v^{\gamma } \comii{D_\eta}^{2s} \Lambda^{m} f}_{L^2}+
       \norm{\comii v^{2s+\gamma}  \Lambda^{m} f}_{L^2}.
  \]
Then combining the above two inequalities we have, with $\eps>0$ arbitrarily small,
\begin{eqnarray*}
  \norm{\com{\Lambda^m,~a}(-\tilde\triangle_v)^s f}_{L^2} \lesssim
  \eps \norm{\comii v^{\gamma } \comii{D_\eta}^{2s} \Lambda^{m} f}_{L^2} +\eps\norm{\comii v^{2s+\gamma }\Lambda^{m} f}_{L^2}+C_\eps \norm{\Lambda^{m}f}_{L^2}.
\end{eqnarray*}
Using the quite similar arguments as above,  we could prove as well that
\begin{eqnarray*}
  \norm{\com{\Lambda^m,~\comii{v}^\gamma}\comii{D_\eta}^{2s} f}_{L^2} \lesssim
    \norm{\comii v^{\gamma } \comii{D_\eta}^{2s} \Lambda^{m} f}_{L^2} + \norm{\comii v^{2s+\gamma }\Lambda^{m} f}_{L^2}+  \norm{\Lambda^{m}f}_{L^2}.
\end{eqnarray*}
The   treatment of other commutators can be handled quite similarly. So we only state the conclusions without proof; that is
\begin{eqnarray*}
  \norm{\com{\Lambda^m,~v\cdot\xi}  f}_{L^2}+ \norm{\com{\Lambda^m,~b}  f}_{L^2}+\norm{\com{\Lambda^m,~\comii v^{2s+\gamma}}  f}_{L^2}     \lesssim \eps \norm{\comii v^{2s+\gamma }\Lambda^{m} f}_{L^2}+C_\eps \norm{\Lambda^{m} f}_{L^2},
\end{eqnarray*}
and
\begin{eqnarray*}
  \norm{\com{\Lambda^m,~\comii v^{\frac{\gamma-2s}{2s+1}}}\comii\tau^{\frac{2s}{2s+1}}  f}_{L^2}+\norm{\com{\Lambda^m,~\comii v^{\frac{\gamma}{2s+1}} }\comii\xi^{\frac{2s}{2s+1}}  f}_{L^2} \lesssim   \norm{\comii v^{2s+\gamma }\Lambda^{m} f}_{L^2}+  \norm{\Lambda^{m} f}_{L^2}.
\end{eqnarray*}
The above four inequalities yield that
\begin{eqnarray}\label{11060301}
  \norm{\com{ \tilde{\mathcal P},~\Lambda^m}  f}_{L^2} \lesssim \eps \norm{\comii v^{\gamma } \comii{D_\eta}^{2s} \Lambda^{m} f}_{L^2} +\eps\norm{\comii v^{2s+\gamma }\Lambda^{m} f}_{L^2}+C_\eps \norm{\Lambda^{m} f}_{L^2}
\end{eqnarray}
since $\com{ \tilde{\mathcal P},~\Lambda^m}=\com{v\cdot\xi,~\Lambda^m}+\com{a,~\Lambda^m}(-\tilde\triangle_v)^s+\com{b,~\Lambda^m}$, and that
\begin{eqnarray*}
  \begin{split}
         &\norm{\Lambda^m\comii{v}^{\frac{\gamma-2s}{1+2s}}\comii{\tau}^{\frac{2s}{1+2s}} f}_{L^2 }+\norm{\Lambda^m \comii{v}^{\frac{\gamma}{1+2s}}\comii{\xi}^{\frac{2s}{1+2s}} f}_{L^2}+\norm{\Lambda^m \comii{v}^{\gamma}\comii{D_v}^{2s}f}_{L^2 }+\norm{\Lambda^m \comii{v}^{2s+\gamma}f}_{L^2 } \\
         &\lesssim \norm{\comii{v}^{\frac{\gamma-2s}{1+2s}}\comii{\tau}^{\frac{2s}{1+2s}}\Lambda^m f}_{L^2 }+\norm{ \comii{v}^{\frac{\gamma}{1+2s}}\comii{\xi}^{\frac{2s}{1+2s}}\Lambda^m f}_{L^2}+\norm{ \comii{v}^{\gamma}\comii{D_v}^{2s}\Lambda^m f}_{L^2 }\\&\quad+ \norm{  \comii{v}^{2s+\gamma} \Lambda^m f}_{L^2 } +  \norm{\Lambda^{m} f}_{L^2}.
\end{split}
\end{eqnarray*}
  As a result the conclusion in Lemma \ref{lem110603} will follow if we could show that
\begin{eqnarray}\label{11060315}
  \begin{split}
           \norm{\comii{v}^{\frac{\gamma-2s}{1+2s}}\comii{\tau}^{\frac{2s}{1+2s}}\Lambda^m f}_{L^2 }&+\norm{ \comii{v}^{\frac{\gamma}{1+2s}}\comii{\xi}^{\frac{2s}{1+2s}}\Lambda^m f}_{L^2}+\norm{ \comii{v}^{\gamma}\comii{D_v}^{2s}\Lambda^m f}_{L^2 }\\  &\qquad\qquad+\norm{  \comii{v}^{2s+\gamma} \Lambda^m f}_{L^2 }
           ~\lesssim ~  \norm{ \Lambda^m \tilde{\mathcal P}  f}_{L^2} +\norm{ \Lambda^m  f}_{L^2}.
\end{split}
\end{eqnarray}
To prove the above inequality we use  the estimate in Proposition  \ref{prp110511} to the function $\Lambda^m  f$; this gives that  the terms on the left hand side  is bounded from above by
\[
   \norm{\tilde{\mathcal  P}  \Lambda^m f}_{L^2} +\norm{ \Lambda^m  f}_{L^2}.
\]
Then  from  \reff{11060301}, it follows that
\begin{eqnarray*}
  \begin{split}
          &\norm{\comii{v}^{\frac{\gamma-2s}{1+2s}}\comii{\tau}^{\frac{2s}{1+2s}}\Lambda^m f}_{L^2 }+\norm{ \comii{v}^{\frac{\gamma}{1+2s}}\comii{\xi}^{\frac{2s}{1+2s}}\Lambda^m f}_{L^2}+\norm{ \comii{v}^{\gamma}\comii{D_v}^{2s}\Lambda^m f}_{L^2 } +\norm{  \comii{v}^{2s+\gamma} \Lambda^m f}_{L^2 }\\  &
            \lesssim   \norm{ \Lambda^m  \tilde{\mathcal P}  f}_{L^2} + \eps\norm{\comii v^{\gamma } \comii{D_\eta}^{2s} \Lambda^{m} f}_{L^2} +\eps\norm{\comii v^{2s+\gamma }\Lambda^{m} f}_{L^2}+C_\eps \norm{ \Lambda^m  f}_{L^2}.
\end{split}
\end{eqnarray*}
Letting $\eps$ small enough gives \reff{11060315}. The proof of  Lemma \ref{lem110603} is thus complete.
\end{proof}

\bigskip
\noindent{\bf Acknowledgements}
The  work was done when the author was a
Postdoctoral Fellow at the Laboratoire de Math\'{e}matiques Jean Leray ,  Universit\'{e} de Nantes,  and he wishes to thank Fr\'{e}d\'{e}ric H\'{e}rau  and Xue Ping Wang for hospitality provided.  The author gratefully acknowledges the support from the Project NONAa of France ( No. ANR-08-BLAN-0228-01),  and the NSF of China  under grant 11001207.








\end{document}